\documentclass[reqno, 12pt]{amsart}
\usepackage{a4wide}
\usepackage[utf8]{inputenc}
\usepackage{amsfonts, amssymb, amsmath, amsthm, mathrsfs, enumerate, enumitem, setspace, url, graphicx, blindtext, scrextend}
\usepackage[mathcal]{euscript} 
\usepackage{tikz, tikz-cd}
	\usetikzlibrary{matrix, arrows, decorations.pathmorphing}

\usepackage{color}
\usepackage{hyperref}
	\definecolor{darkred}{rgb}{0.5,0,0}
	\definecolor{darkgreen}{rgb}{0,0.5,0}
	\definecolor{darkblue}{rgb}{0,0,0.7}
	\hypersetup{linkcolor=darkblue,filecolor=darkgreen,urlcolor=darkred,citecolor=darkblue}

\usepackage[toc,page]{appendix}
\newcommand{\stoptocwriting}{%
	\addtocontents{toc}{\protect\setcounter{tocdepth}{-5}}}
\newcommand{\resumetocwriting}{%
	\addtocontents{toc}{\protect\setcounter{tocdepth}{\arabic{tocdepth}}}}

\theoremstyle{plain}
\newtheorem{theorem}{Theorem}[section]
\newtheorem*{theorem*}{Theorem}
\newtheorem{proposition}[theorem]{Proposition}
\newtheorem{lemma}[theorem]{Lemma}
\newtheorem{corollary}[theorem]{Corollary}

\newtheorem*{hypothesis h}{Hypothesis H}
\newtheorem*{hypothesis h1}{Hypothesis H$_1$}

\theoremstyle{definition}

\newtheorem{example}[theorem]{Example}

\theoremstyle{remark}
\newtheorem{remark}[theorem]{Remark}
\newtheorem*{acknowledgements}{Acknowledgements}

\numberwithin{equation}{section}

\newcommand{\NN}{\mathbb{N}}
\newcommand{\ZZ}{\mathbb{Z}}
\newcommand{\QQ}{\mathbb{Q}}
\newcommand{\Zp}{\mathbb{Z}_p}

\newcommand{\Qp}{\mathbb{Q}_p}

\newcommand{\RR}{\mathbb{R}}

\newcommand{\FF}{\mathbb{F}}
\newcommand{\Fp}{\mathbb{F}_p}

\renewcommand{\AA}{\mathbb{A}}

\newcommand{\A}{\mathbf{A}}
\newcommand{\X}{\mathcal{X}}
\newcommand{\x}{\mathbf{x}}
\newcommand{\veps}{\varepsilon}

\def\gra{{\mathfrak a}} 
\def\grc{{\mathfrak c}} 
\def\gri{{\mathfrak i}} \def\grI{{\mathfrak I}}
\def\grj{{\mathfrak j}} 
\def\grk{{\mathfrak k}} 
\def\grp{{\mathfrak p}} \def\grP{{\mathfrak P}}
\def\grq{{\mathfrak q}} 
\def\grr{{\mathfrak r}}

\DeclareMathOperator{\Br}{Br}
\DeclareMathOperator{\inv}{inv}
\DeclareMathOperator{\ev}{ev}
\DeclareMathOperator{\Spec}{Spec}

\DeclareMathOperator{\valp}{v_{\it p}}

\DeclareMathOperator{\N}{N}
\DeclareMathOperator{\ClK}{Cl^+(K)}
\DeclareMathOperator{\Gal}{Gal}

\renewcommand{\(}{\left(}
\renewcommand{\)}{\right)}
\newcommand{\dmid}{\parallel}
\newcommand{\nequiv}{\not\equiv}

\title{Integral points on generalised affine Ch\^{a}telet surfaces}
\author{Vladimir Mitankin}
	\address{Max Planck Institut f\"{u}r Mathematik \\
	Vivatsgasse 7 \\
	53111 Bonn \\
	Germany
	}
	\email{vmitankin@mpim-bonn.mpg.de}
\date{\today}
 
\begin{document}
\begin{abstract}
	We show, conditionally on Schinzel's hypothesis, that the only obstruction to the integral Hasse principle for generalised affine Ch\^{a}telet surfaces is the Brauer--Manin one.
\end{abstract}
\maketitle
\tableofcontents

\section{Introduction}
This article is dedicated to the study of the set of integral points on an integral model of the affine surfaces $X$ over $\QQ$ of the shape
\begin{equation} 
	\label{def:affine Chatelet}
	X : \quad \N_{\QQ(\sqrt{a})}(x, y) = P(t).
\end{equation}
In the current set-up $a \in \ZZ$ is non-zero squarefree, $P(t) \in \ZZ[t]$ is separable as a polynomial over $\QQ$ and $\N_{\QQ(\sqrt{a})}(x, y) = \N_{\QQ(\sqrt{a})/\QQ}(x + \omega y)$ is the norm form of $\QQ(\sqrt{a})$ for the basis $1, \omega$ over $\QQ$ with $\omega = -(1 + \sqrt{a})/2$ when $a \equiv 1 \bmod 4$ and $\omega = \sqrt{a}$ when $a \equiv 2, 3 \bmod 4$. We shall refer to $X$ as a \emph{generalised affine Ch\^{a}telet surface}. Let $\X$ be the integral model of $X$ over $\ZZ$ given by the same equation, $\X(\ZZ)$ its set of integral points and $\X(\A_\ZZ) = X(\RR) \times \prod_p \X(\Zp)$ its set of adeles. One can embed $\X(\ZZ)$ diagonally in $\X(\A_\ZZ)$ and hence a necessary condition for the existence of integral points on $\X$ is $\X (\A_{\ZZ}) \neq \emptyset$. We say that $X$ \emph{satisfies the integral Hasse principle} if this is also a sufficient condition.

Colliot-Th\'el\`ene and Xu \cite[\S1]{CTX09} found that the Brauer--Manin obstruction \cite{Man74} plays an important role when it comes to the existence of integral points on $\X$. More precisely, given a $\QQ$-variety $X$ there is a pairing between the Brauer group $\Br X$ of $X$ and the set of $\QQ$-adeles $X(\A_\QQ)$ of $X$. This pairing allowed Manin to construct a sequence of inclusions $X(\QQ) \subseteq X(\A_\QQ)^{\Br X} \subseteq X(\A_\QQ)$, where the intermediate set is known as the Brauer--Manin set. A Brauer--Manin obstruction to the existence of rational points on $X$ is then present if $X(\A_{\QQ}) \neq \emptyset$ but $X(\A_\QQ)^{\Br X} = \emptyset$. It is said to be  the only obstruction to the Hasse principle if $X(\A_\QQ)^{\Br X} \neq \emptyset$ implies $X(\QQ) \neq \emptyset$. A similar sequence of inclusions $\X(\ZZ) \subseteq \X(\A_\ZZ)^{\Br X} \subseteq \X(\A_\ZZ)$ is available. It induces a Brauer--Manin obstruction to the integral Hasse principle, an overview of which is included in section~\ref{sec:BMO}.

The surface $X$ given in \eqref{def:affine Chatelet} is closely related to a Ch\^{a}telet surface. The latter is a smooth proper model of $X$ when $\deg P(t) = 3$ or $4$. The arithmetic of a Ch\^{a}telet surface is very well understood by the work of Colliot-Th\'el\`ene, Sansuc and Swinnerton-Dyer \cite{CTSSD87a}, \cite{CTSSD87b}. They have shown that the only obstruction to the Hasse principle and weak approximation for a Ch\^{a}telet surface is the Brauer--Manin one. Under the following hypothesis this is still true for a smooth proper model of $X$ when the restriction on the degree of $P(t)$ is dropped, as the results of Colliot-Th\'el\`ene and Sansuc \cite{CTS82}, Swinnerton-Dyer \cite{SD94}, Serre \cite{Ser92}, and Colliot-Th\'el\`ene and Swinnerton-Dyer \cite{CTSD94} show. 

\begin{hypothesis h}[Schinzel's hypothesis] \label{hyp:H}
	Let $P_1(t), \dots, P_n(t) \in \ZZ[t]$ be irreducible polynomials over $\ZZ$ with positive leading coefficients and such that the set of values $\prod_{i = 1}^{n} P_i(m)$ for $m$ running in $\ZZ$ has no fixed prime divisor. Then there exist infinitely many natural numbers $m \in \NN$ such that each of $P_1(m), \dots, P_n(m)$ is a prime.
\end{hypothesis h}

Unconditional results when $P(t)$ is of arbitrary degree were attained by Browning, Matthiesen and Skorobogatov \cite{BMS14} if $P(t) \in \QQ[t]$ is a product of linear factors over $\QQ$. In this setting a variant of Hypothesis~\hyperref[hyp:H]{H} is available via recent developments in additive combinatorics.

Let $X$ be as in \eqref{def:affine Chatelet} with $P(t) = \prod_{i = 1}^n P_i(t)$, where $P_i(t) \in \ZZ[t]$ are irreducible over $\ZZ$ and pairwise coprime polynomials. As explained in section~\ref{sec:BMO} for each $i \in \{1, \dots, n\}$ the quaternion algebra $(a, P_i(t))$ belongs to $\Br X$ and thus the set $\X(\A_{\ZZ})^{\mathcal{A}}$ of adeles orthogonal to $\mathcal{A}$ contains $\X(\ZZ)$. We are now ready to state our main result.

\begin{theorem} 
\label{thm:Chatelet} 
	Let $a \in \ZZ$ be non-zero squarefree and let $P_1(t), \dots, P_n(t) \in \ZZ[t]$ be irreducible such that $P(t) = \prod_{i = 1}^n P_i(t)$ is separable as a polynomial over $\QQ$. Assume that $P_1(t)$ is linear or that the discriminant of the splitting field of $P_1(t)$ is coprime to the discriminant of $\QQ(\sqrt{a})$. Define $X$ by \eqref{def:affine Chatelet} and let $\mathcal A = \langle (a, P_1(t)), \dots, (a, P_n(t)) \rangle$ be the subgroup of $\Br X$ generated by the quaternion algebras $(a, P_i(t))$. Assume that the image of the natural projection $\X (\A_\ZZ)^{\mathcal A} \rightarrow X(\RR)$ is not bounded. Then, conditionally on Hypothesis~\hyperref[hyp:H]{H}, we have 
 	\[
 		\X (\A_\ZZ)^{\mathcal A} \neq \emptyset 
 		\implies \X(\ZZ) \neq \emptyset.
 	\]
 	Moreover, let $S$ be a finite set of finite places of $\QQ$. Then for each $\veps > 0$ and for each adelic point $(x_p, y_p, t_p) \in \X (\A_\ZZ)^{\mathcal A}$ with $(x_\infty, y_\infty, t_\infty)$ in an unbounded component of $X(\RR)$ there exists a point $(x, y, t) \in \X(\ZZ)$ such that $|t - t_p|_p < \veps$ for any $p \in S$. 
\end{theorem}

In the case $a = -1$ and $P(t) = - t^k + m$ the surfaces $X$ were already studied by Gundlach \cite{Gun13}, again conditionally on Hypothesis~\hyperref[hyp:H]{H}. However, our methods and Gundlach's ideas are different in nature. We continue with a few remarks regarding the conditions assumed in Theorem~\ref{thm:Chatelet}.

There are only finitely many points with integral coordinates inside the image of the projection $\X (\A_\ZZ)^{\mathcal A} \rightarrow X(\RR)$ when it is bounded. If none of them satisfies \eqref{def:affine Chatelet} the inclusion $\X(\ZZ) \hookrightarrow X(\RR)$ yields $\X(\ZZ) = \emptyset$ and then we say that there is an obstruction to the existence of integral points coming from the archimedean place. Berg \cite[Prop.~1.4]{B17} points out that if one takes $a = -1$ and $P(t) = -t^4 + m$ with $m = 33, 43, 67, \dots$ then there is no Brauer--Manin obstruction, however, $\X(\ZZ) = \emptyset$ because of the obstruction at $\infty$. One can produce many examples in a similar manner where the lack of integral points on $\X$ is explained via the obstruction at the real place. Thus a certain unboundedness condition is absolutely necessary. Such a condition features in Theorem~\ref{thm:Chatelet}.

If one drops the requirement that $P_1(t)$ is linear or the discriminant of its splitting field is coprime to the discriminant of $\QQ(\sqrt{a})$, then there may be extra elements of $\Br X$ that need to be taken into account. Such an example is $a = -5$ and  $P(t) = t^5 + 20t + 32$, where $\Br X/\Br \QQ \cong \ZZ/5\ZZ$ is generated by a non-cyclic algebra of order $5$ \cite[Thm~1.2.]{B17}. However, Remark~\ref{rem:class nb 1 or 2} allows us to forget this assumption when the narrow class number of $\QQ(\sqrt{a})$ is $1$ or $2$.

By \eqref{eq:extended inclusions} we see that Theorem~\ref{thm:Chatelet} implies that the Brauer--Manin obstruction is the only obstruction to the integral Hasse principle for $X$, conditionally on Hypothesis~\hyperref[hyp:H]{H}. As explained in section~\ref{sec:BMO}, when $P(t)$ is irreducible the quaternion algebra $(a, P(t))$ is a trivial element of $\Br X$ and thus $\X (\A_\ZZ)^{\mathcal A} = \X (\A_\ZZ)$. Then the statement of Theorem~\ref{thm:Chatelet} becomes equivalent to $X$ satisfying the integral Hasse principle. 

In general, it is seldom an easy task to give an explicit description of the elements of $\Br X$ which is, usually, the first non-trivial step in the analysis of $\X(\A_{\ZZ})^{\Br}$. Progress in this direction for generalised affine Ch\^{a}telet surfaces was made by Berg, who describes the quotient $\Br X / \Br \QQ$ in the case when $P(t)$ is irreducible over $\QQ$ with additional assumptions on its splitting field \cite[Thm~1.1. and Thm~1.2.]{B17}. Our method has the advantage that it requires only basic knowledge of $\Br X$. We work with the quaternion algebras $(a, P_1(t)), \dots, (a, P_n(t)) \in \Br X$. This allows us to study infinitely many surfaces which are not covered in Berg's work, for example those with $\QQ(\sqrt{a})$ not contained in the splitting field of $P(t)$. Other difficulties were pointed out by Colliot-Th\'el\`ene and Harari \cite[p.~175]{CTH16}. The projection to the $t$-coordinate endows $X$ with a conic bundle structure over the affine line. The two main problems with the currently known techniques for studying integral points according to them are that for any smooth fibre $X_t$ the quotient $\Br X_t / \Br \QQ$ is infinite and the fibres corresponding to the zeros of $P(t)$ are not split.

Our strategy exploits the structure of the narrow ideal class group of $\QQ(\sqrt{a})$ to adapt a method of Colliot-Th\'el\`ene and Sansuc \cite{CTS82} to this setting. More precisely, let $X$ be as in \eqref{def:affine Chatelet}. Under the assumption of Hypothesis~\hyperref[hyp:H]{H} we reduce the study of the existence of points on $\X$ to understanding if a suitable fibre $X_t$ over an integral $t$ with $X_t(\QQ) \neq \emptyset$ has an integral point. This is answered positively with the help of the next proposition.

\begin{proposition}
\label{prop:rational imples integral}
	Let $a$ be a non-zero squarefree integer and let $K = \QQ(\sqrt{a})$. Let $S$ be a finite set of finite primes of $\ZZ$ which are not inert in $O_K$. Let $\grp_1, \dots, \grp_{h-1}$ be prime ideals of $O_K$ above split primes $p_1, \dots, p_{h-1} \not\in S$, respectively, such that the non-trivial elements of the narrow class group of $K$ are precisely the distinct classes $[\grp_1], \dots, [\grp_{h-1}]$. Then if
	\begin{equation} 
		\label{eq:rat -> int sol}
		\N_K(x, y)
		= p_1 \dots p_{h-1} \prod_{p \in S} p
	\end{equation}
	has a solution with $x, y \in \QQ$, it also has a solution with $x, y \in \ZZ$. 
\end{proposition}

One can view Proposition~\ref{prop:rational imples integral} as a separate result, a corollary of which is that certain affine conics satisfy the integral Hasse principle. 

\stoptocwriting
\subsection{Examples}
Keep notation as in Theorem~\ref{thm:Chatelet}. We conclude the introduction with two numerical examples which demonstrate how Theorem~\ref{thm:Chatelet} with some of its assumptions relaxed still yields the existence of integral points on $\X$. 

\begin{example}
	\label{exp:narrow 1 or 2}
	Let $X \subset \AA^3$ be the affine surface defined over $\QQ$ by
	\[
		X: \quad x^2 - 2y^2 = t^5 + 6t^2 + 3t + 3.
	\]
	Eisenstein's criterion verifies that $P(t) = t^5 + 6t^2 + 3t + 3$ is irreducible and thus $\X(\A_\ZZ)^{\mathcal{A}} = \X(\A_{\ZZ})$ as explained in section~\ref{sec:BMO}. Hence it suffices to verify that $\X(\A_{\ZZ}) \neq \emptyset$. It is clear that $X(\RR) \neq \emptyset$. Hensel's lemma and the existence of non-trivial solution to $x^2 - 2y^2 \equiv 3 \bmod p$ imply that $\X(\Zp) \neq \emptyset$ for each $p \neq 2, 3$. When $p = 2$ or $3$ the solutions $(1, 1, 4) \bmod 8$ and $(1, 0, 1) \bmod 3$ to $x^2 - 2y^2 \equiv P(t) \bmod 8$ and mod $3$, respectively, and Hensel's lemma show the existence of a point in $\X(\Zp)$. Thus $\X(\A_{\ZZ}) \neq \emptyset$. The narrow class group of $\QQ(\sqrt{2})$ is trivial. Thus by Remark~\ref{rem:class nb 1 or 2} we are in position to apply Theorem~\ref{thm:Chatelet} which states, conditionally on Hypothesis~\hyperref[hyp:H]{H}, that $\X$ has an integral point. In fact, $(59, 47, -4) \in \X(\ZZ)$.
\end{example}

In the second example Theorem~\ref{thm:Chatelet} is applied without the assumption of Hypothesis~\hyperref[hyp:H]{H}.

\begin{example}
	\label{exp:unconditional}
	Let $a \equiv 1 \bmod 8$ be a positive squarefree integer such that the narrow class number of $\QQ(\sqrt{a})$ is $1$ or $2$, e.g. $a  = 17, 33, 41, 57, 73, 89, 97, \dots$ Let $X \subset \AA^3$ be the affine surface over $\QQ$ given by 
	\[
		X: \quad x^2 - xy + \frac{1 - a}{4}y^2 = \prod_{i = 1}^n (t^{2k_i} - a^{2m_i - 1}),
	\]
	where $n$, $k_i$ and $m_i$ are positive integers such that $t^{2k_i} - a^{2m_i + 1}$ are irreducible over $\ZZ$ and pairwise coprime. We claim that $\X(\A_\ZZ)^{\mathcal{A}} \neq \emptyset$. Indeed, since $a > 0$ we have $X(\RR) \neq \emptyset$ and $(a, t_\infty^{2k_i} - a^{2m_i - 1})_\infty = 1$ for any real point $(x_\infty, y_\infty, t_\infty)$. The later implies the unboundedness condition in Theorem~\ref{thm:Chatelet}. One easily sees that $\X(\Zp) \neq \emptyset$ for any $p \nmid a$ odd by a similar argument as in Example~\ref{exp:narrow 1 or 2} via Remark~\ref{rem:equivalence}. Moreover, for each such $p$ either $a \in \Fp^{* 2}$ or $\valp(t_p^{2k_i} - a^{2m_i - 1}) = 0$ for any $(x_p, y_p, t_p) \in \X(\Zp)$. Thus $(a, t_p^{2k_i} - a^{2m_i - 1})_p = 1$. If $p \mid a$ is odd, then $(1, 0, 1) \bmod p$ lifts to a $\Zp$-point by Hensel's lemma. For this point the corresponding Hilbert symbols for the quaternion algebras $(a, t^{2k_i} - a^{2m_i - 1})$ are all equal to 1. Finally, $a \equiv 1 \bmod 8$ implies that $(1 - a)/4$ is even and thus depending on the parity of $n$ one of $(1, 0, 0)$ or $(1, 2, 0) \bmod 8$ lifts to a point in $\X(\ZZ_2)$ by Hensel's lemma. Furthermore, $a \equiv 1 \bmod 8$ implies that the corresponding Hilbert symbols are again all 1 at $p = 2$. This gives a point $(x_p, y_p, t_p) \in \X(\A_\ZZ)^{\mathcal{A}}$. Let a finite set $S$ of finite primes be given which we extend as in the proof of Theorem~\ref{thm:Chatelet}. Instead of Hypothesis~\hyperref[hyp:H]{H} we can use here the fact that $\ZZ$ satisfies strong approximation away from $\infty$. Thus there is $\lambda \in \ZZ$ arbitrarily close to $t_p$ when $p \in S$ and such that
	\[
		\prod_{i = 1}^n (\lambda^{2k_i} - a^{2m_i - 1})
		= \prod_{p \in S'} p^{n_p},
	\]
	where $S'$ is some finite set containing $S$. All we need to guarantee now is that no prime $p$ of $\ZZ$ which is inert in $\QQ(\sqrt{a})$ divides $\prod_{i = 1}^n (\lambda^{2k_i} - a^{2m_i - 1})$. However, this is a straightforward corollary of the shape of each of the polynomials $t^{2k_i} - a^{2m_i - 1}$, the fact that for odd such primes $a$ is not a quadratic residue mod $p$ and $2$ is split in $\QQ(\sqrt{a})$. Following the same circle of ideas as in the proof of Theorem~\ref{thm:Chatelet} puts us in position where Proposition~\ref{prop:rational imples integral} is applicable. Hence $\X(\ZZ) \neq \emptyset$.
\end{example}

\subsection*{Structure}
This article is organised as follows. In section~2 we give an overview of the Brauer--Manin obstruction for integral points. Section~3 is dedicated to the proof of Proposition~\ref{prop:rational imples integral}. In section~4 we state and prove Proposition~\ref{prop:Chebotarev} and its Corollary~\ref{cor:Chebotarev} which are needed in the proof of Therorem~\ref{thm:Chatelet} when the narrow class number of $\QQ(\sqrt{a})$ is greater than 2. In the last section we prove Theorem~\ref{thm:Chatelet}.

\subsection*{Notation}
Throughout the rest of this article the following convention will be in force. For any number field $k$ denote its ring of integers by $O_k$ and its discriminant by $d_k$. We preserve $a$ for a fixed non-zero squarefree integer. Let $K = \QQ(\sqrt{a})$ and let $\ClK$ be the narrow class group of $K$. The symbol $[\cdot]$ denotes an element $\ClK$ and we shall write $h = \# \ClK$ for the narrow class number of $K$. For any $b \in K^{\times}$ let $(b)$ denote the principal fractional ideal $b O_K$. The subscript $i$ will usually be preserved for an irreducible factor of $P(t)$ while the subscript $j$ will usually be used for one of the primes $p_1, \dots, p_{h-1}$ appearing in Proposition~\ref{prop:rational imples integral}
\resumetocwriting

\begin{acknowledgements}
	An earlier version of this work is part of the author's PhD thesis. The author is very grateful to his supervisor Tim Browning for his constant support and guidance. The author thanks Roger Heath-Brown for suggesting the use of Proposition~\ref{prop:rational imples integral} and Martin Bright, Jean-Louis Colliot-Th\'el\`ene, Yonatan Harpaz, Daniel Loughran, Carlo Pagano, Arne Smeets, Efthymios Sofos, Jesse Thorner and Ulrich Derenthal for many useful discussions and comments. While working on this paper the author was supported by ERC grant \texttt{306457}. Part of this work was done while in residence at MSRI, the author would like to express his gratitude for their hospitality. The author is also grateful to Max Planck Institute for Mathematics in Bonn for its hospitality and financial support.
\end{acknowledgements}

\section{The Brauer--Manin obstruction}
\label{sec:BMO}
This section contains a brief overview of the Brauer--Manin obstruction for integral points, as introduced by Colliot-Th\'el\`ene and Xu \cite[\S 1]{CTX09}. Let $\X$ be a separated scheme of finite type over $\ZZ$. Let $X = \X \times_\ZZ \QQ$ and denote by $\Br X = H_{\text{\'{e}t}}^2(X, \mathbb{G}_m)$ the Brauer group of $X$. We use standard notation for the completions of $\QQ$ and $\ZZ$ with the convention that $\ZZ_\infty = \RR$.

Denote by $X(\A_\QQ)$ the set of adeles of $X$. Its elements are points $(\x_p) \in \prod_{p \le \infty} X(\Qp)$ such that, for all but finitely many primes $\x_p$ belongs to $\X(\Zp)$. Thus the set of rational points $X(\QQ)$ on $X$ can be embedded diagonally in $X(\A_\QQ)$ and we identify $X(\QQ)$ with its image under this embedding.

For each $\alpha \in \Br X$ and for each prime $p \le \infty$ class field theory yields evaluation maps $\ev_{\alpha} : X(\Qp) \rightarrow  \Br \Qp$.  Moreover, there is a group homomorphism $\inv_p : \Br \Qp \rightarrow \QQ / \ZZ$ and thus a map
\begin{equation} 
	\label{eq:ev composed with inv}
	\begin{tikzcd}
		X(\Qp) \ar[r, "\ev_{\alpha}"]  
		&\Br \Qp \ar[r, "\inv_p"] 
		&\QQ / \ZZ.
	\end{tikzcd}
\end{equation}
It gives rise to a natural pairing, the Brauer--Manin pairing, between $X (\A_\QQ)$ and $\Br X$ defined via
\begin{equation} 
	\label{def:Brauer-Manin pairing}
	\begin{split}
		X(\A_\QQ) \times \Br X 
		&\longrightarrow 
		\QQ / \ZZ, 
		\\
		\((\x_p), \alpha \) 
		&\longmapsto 
		\sum_{p \le \infty} \inv_p \( \ev_{\alpha} \( \x_p \)\).
	\end{split}
\end{equation}
The key point is that we have a commutative diagram 
\[
	\begin{tikzcd}[column sep=large]
		&X(\QQ) \arrow[hookrightarrow]{r} \arrow[d, swap, "\ev_{\alpha}"]
		&X(\A_{\QQ}) \arrow[d, "\ev_{\alpha}"] \\
		&\Br \QQ \arrow[r]
		&\underset{p \le \infty}{\bigoplus}\Br \Qp \arrow[r, "\sum_p \inv_p"]
		&\QQ / \ZZ
	\end{tikzcd}
\]
with the bottom line fitting into a short exact sequence by class field theory. Thus any rational point on $X$ lies is in the left kernel of the pairing \eqref{def:Brauer-Manin pairing}.

The Brauer--Manin obstruction for integral points can be described in the following way. Recall that $\X(\A_\ZZ) = X(\RR) \times \prod_p \X(\Zp)$. By separateness of $\X$ one can view $\X(\ZZ)$ and $\X(\A_\ZZ)$ as subsets of $X(\QQ)$ and $X(\A_\QQ)$, respectively. Therefore, there is an induced pairing
\begin{equation} 
	\label{def:integral BM pairing}
	\X(\A_\ZZ) \times \Br X
	\rightarrow \QQ / \ZZ,
\end{equation}
which vanishes on the image of $\X(\ZZ)$ in $\X(\A_\ZZ)$. For any $\alpha \in \Br X$ let $\X(\A_\ZZ)^{\alpha}$ denote the set of points in $\X(\A_\ZZ)$ orthogonal to $\alpha$ under \eqref{def:integral BM pairing}. Define the Brauer-Manin set $\X(\A_\ZZ)^{\Br X}$ as the left kernel of \eqref{def:integral BM pairing}, i.e.
\begin{equation}
	\label{def:integral Brauer-Manin set}
	\X(\A_\ZZ)^{\Br X}
	=	\bigcap_{\alpha \in \Br X} \X(\A_\ZZ)^{\alpha}.
\end{equation} 

It follows that we have inclusions
\[
	\X(\ZZ)
	\subseteq \X(\A_\ZZ)^{\Br X} 
	\subseteq	\X(\A_\ZZ).
\]
A Brauer--Manin obstruction to the integral Hasse principle for $X$ is present if the set of adeles $\X(\A_\ZZ)$ is non-empty but the Brauer--Manin set $\X(\A_\ZZ)^{\Br X}$ is empty. Alternatively, the Brauer--Manin obstruction is said to be the only obstruction to the integral Hasse principle if $\X(\A_\ZZ)^{\Br X} \neq \emptyset$ is sufficient to imply the existence of an integral point on $\X$.

Let $\mathcal{A}$ be a subgroup of $\Br X$ and let $\X(\A_\ZZ)^{\mathcal{A}}$ be the set of adeles orthogonal to the elements of $\mathcal{A}$ under \eqref{def:integral BM pairing}. Thus by \eqref{def:integral Brauer-Manin set} the Brauer--Manin set is a subset of $\X(\A_\ZZ)^{\mathcal{A}}$ and we have a further chain of inclusions 
\begin{equation}
	\label{eq:extended inclusions}
	\X(\ZZ)
	\subseteq \X(\A_\ZZ)^{\Br X} 
	\subseteq \X(\A_\ZZ)^{\mathcal{A}} 
	\subseteq \X(\A_\ZZ).
\end{equation}

In this article we are interested in $X$ defined by \eqref{def:affine Chatelet} with $P(t) = \prod_{i = 1}^{n} P_i(t) \in \ZZ[t]$ the factorisation of $P(t)$ into irreducible polynomials of $\ZZ[t]$. The quaternion algebras $(a, P_i(t))$ for $i = 1, \dots, n$ belong to $\Br X$ and we take $\mathcal{A}$ to be the subgroup of $\Br X$ generated by them.  Indeed, let us show that $(a, P_i(t)) \in \Br X$. Since $\N_{K}(x, y) = \N_{K/Q}(x + \omega y)$, then $(a, \N_{K}(x, y))$ is clearly trivial in $\Br (\QQ(X))$ and thus we have
\[
	(a, P_i(t))
	= \prod_{j \neq i} (a, P_j(t)) 
\]
in the Brauer group of the function field of $X$. On the other hand, $(a, P_i(t))$ is an Azumaya algebra on the open subvariety $U$ given by $P_i(t) \neq 0$ while $\prod_{j \neq i} (a, P_j(t))$ is Azumaya on $V$ given by $\prod_{j \neq i} P_j(t) \neq 0$. Moreover, the above verifies that these two quaternion algebras are compatible on $U \cap V$. The condition $P(t)$ separable implies that $X = U \cup V$ is smooth and thus $\Br X$ injects into $\Br (\QQ(X))$ and $(a, P_i(t))$ lies in the image of this injection.

\section{Integral points on affine conics}
\label{sec:affine conics}
In this section we prove Proposition~\ref{prop:rational imples integral}. With notation as in its statement assume that \eqref{eq:rat -> int sol} has a solution $(x, y) = \(\frac{m}{k}, \frac{\ell}{k}\)$ over $\QQ$. Therefore
\begin{equation} 
	\label{eq:rational point on Chatelet}
	\N_K(m, \ell) 
	= k^2 p_1 \dots p_{h-1} \prod_{p \in S} p.
\end{equation}
We can further assume that $(m, \ell) = 1$. Indeed, if a prime $p$ divides $(m, \ell)$, then $p^2$ divides the right hand side and since $\{p_1, \dots, p_{h-1}\} \cap S = \emptyset$, then $p \mid k$. Thus we can divide $m$, $l$ and $k$ through the common factor $p$. 

Recall that we took a basis $1, \omega$ of $K$ over $\QQ$ with $\omega  = -(1 +\sqrt{a})/2$ when $a \equiv 1 \bmod 4$ and $\omega  = \sqrt{a}$ when $a \equiv 2,3 \bmod 4$. Factoring \eqref{eq:rational point on Chatelet} into ideals of $O_K$ yields
\[
	\gri \gri'
	= (k)^2 \grp_1 \grp_1' \dots \grp_{h-1} \grp_{h-1}' \prod \grp \grp',
\]
where $'$ denotes the conjugate of an ideal in $O_K$, $\gri = (m + \omega\ell)$ and the product is taken over the corresponding prime ideals above the primes in $S$.

If a prime $\grp'$ from the product $\prod \grp \grp'$ divides $\gri$, then we swap the labels of $\grp$ and $\grp'$ so that it is $\grp$ which divides $\gri$. Let $\mathcal{S}$ denote the set of primes $\grp$ appearing in the product $\prod \grp \grp'$ which divide $\gri$. Thus there is a prime ideal $\grp \in \mathcal{S}$ above $p$ for each prime $p \in S$ and the ideals in $\mathcal{S}$ are pairwise coprime and non-conjugate. Since $\grp_1, \dots, \grp_{h-1}$ are pairwise coprime and coprime to the ideals in $\mathcal{S}$, then
\[
	\gri
	= \grj \grr_1 \dots \grr_{h-1} \prod_{\grp \in \mathcal{S}} \grp,
\]
where $\grr_j$ is either $\grp_j$ or $\grp_j'$ for all $j = 1, \dots, h-1$ and $\grj \grj'= (k)^2$.

We claim that $[\grj]$ is a square in $\ClK$. To show this we analyse the common divisors of $\gri$ and $\gri'$. If $\grq$ is a prime factor of $\gri$ and $\gri'$, then $m + \omega\ell$ and $m + \overline{\omega}\ell$ are elements of $\grq$, where $\overline{\omega}$ denotes the conjugate of $\omega$. Their sum and their difference are in $\grq$ as well, whence $a(2m - \ell), 2a\ell \in \grq$ when $a \equiv 1 \bmod 4$ and $2am, 2a\ell \in \grq$ when $a \equiv 2,3 \bmod 4$. It follows that $\grq \mid (2a)$ since $(m,\ell) = 1$. 

If $a \equiv 2,3 \bmod 4$ the discriminant $d_K$ of $K$ equals $4a$. Thus $\grq$ lies over a prime which divides $d_K$ but all such primes ramify in $O_K$. If a ramified prime divides $k$ the ideal above it must appear to an exponent which is a multiple of $4$ in the ideal prime factorisation of $(k)^2$. Since we took $\grj \grj' = (k)^2$, then a standard application of the unique prime ideal factorization theorem shows that $\grj = \grk^2$ and $\grj'= \grk'^2$ for some ideal $\grk$ of $O_K$.

If $a \equiv 1 \bmod 4$, then $d_K = a$. In this case $\grq$ either lies over an odd prime dividing $a$ and hence is ramified in $O_K$ or $\grq$ divides $2$. In the former case or if $a \equiv 1 \bmod 8$ where $2$ is split the same argument as before shows that $\grj = \grk^2$. If, however, $a \equiv 5 \bmod 8$ and $\grq \mid (2)$, then $\grq = (2)$ since $2$ is inert in $O_K$. Thus $\grq$ is a principal ideal of $O_K$ generated by an element of positive norm which implies that its class in $\ClK$ is trivial. Hence $[\grj] = [\grk]^2$. 

We have shown that the following equality in the narrow class group $\ClK$ holds:
\[
	[\gri]
	= [\grk^2 \grr_1 \dots \grr_{h-1} \prod_{\grp \in \mathcal{S}} \grp]. 
\]
The class of $\gri$ is trivial since $\N_{K/\QQ}(m + \omega\ell) > 0$ by \eqref{eq:rational point on Chatelet}. We deduce that 
\begin{equation}
	\label{eq:class eq}
	[\grr_1 \dots \grr_{h-1} \prod_{\grp \in \mathcal{S}} \grp]
	= \([\grk]^{-1}\)^2.
\end{equation}
Swapping $\grr_j$ and $\grr_j'$ for some $j \in \{1, \dots, h-1\}$ above will preserve the fact that the class on the left hand side is a square in $\ClK$. Indeed $[\grr_j \grr_j'] = [O_K]$ and so
\[
	[\grr_1 \dots \grr_{j - 1} \grr_j' \grr_{j + 1} \dots \grr_{h-1} \prod_{\grp \in \mathcal{S}} \grp]
	= \([\grk]^{-1}[\grr_j']\)^2.
\]
Therefore 
\[
	[\grp_1 \dots \grp_{h-1} \prod_{\grp \in \mathcal{S}} \grp]
	= [\grc]^2
\]
for some ideal $\grc$ of $O_K$. Since $[\grp_1], \dots, [\grp_{h-1}]$ are all of the non-trivial elements of $\ClK$, then $[\grc] = [\grp_j]$ for some $j \in \{1, \dots, h-1\}$ or $[\grc]$ is trivial. In the former case we can multiply both sides of the last equation above by $[\grc]^{-2}$. As above this is equivalent to substituting $[\grp_j]$ with $[\grp_j']$ in the left hand side. In this way we get a trivial class
\[
	[\grp_1 \dots \grp_{j - 1} \grp_j' \grp_{j + 1} \dots \grp_{h-1} \prod_{\grp \in \mathcal{S}} \grp]
	= \([\grc][\grc]^{-1}\)^2
	= [O_K].
\]
When $[\grc]$ is trivial the same conclusion is still valid if instead of $\grp'$ one has $\grp$ above. 

In both cases the assumptions of Proposition~\ref{prop:rational imples integral} imply that all prime ideals on the left hand side lie above primes of $\ZZ$ which are not inert in $O_K$. Therefore the norm of the ideal on the left hand side is
\[
	\N ( \grp_1 \dots \grp_{j - 1} \grp_j' \grp_{j + 1} \dots \grp_{h-1} \prod_{\grp \in \mathcal{S}} \grp )
	= p_1 \dots p_{h-1} \prod_{p \in S} p.
\]

We have shown that there is a principal fractional ideal $\grI$ of $O_K$ whose class in $\ClK$ is trivial and whose norm is $p_1 \dots p_{h-1} \prod_{p \in S} p$. The triviality of $[\grI]$ implies that $\grI$ has a generator $m_* + \omega\ell_*$ with positive norm. But then $\N(\grI) = \N_{K/\QQ}(m_* + \omega\ell_*)$. In other words, there are integers $m_\ast, \ell_\ast$ such that
\begin{equation}
	\label{eq:integral point}
	N_K(m_*, \ell_*)
	= p_1 \dots p_{h-1} \prod_{p \in S} p.
\end{equation}
This concludes the proof of Proposition~\ref{prop:rational imples integral}.
\qed

\begin{remark} 
	\label{rem:class nb 1 or 2}
	If $h = 1$, then the set $\{p_1, \dots, p_{h-1}\}$ is clearly empty. If $h = 2$, then the class of $\grk^2$ is trivial in $\ClK$ and hence already \eqref{eq:class eq} implies that $[\grp_1 \prod_{\grp \in \mathcal{S}} \grp]$ is trivial as well. In this case our argument proves the existence of a principal fractional ideal of $O_K$ of norm the right hand side of \eqref{eq:rat -> int sol} and hence the existence of an integral solution to \eqref{eq:rat -> int sol}. What is important is that we need no assumption on any prime appearing on the right hand side of \eqref{eq:rat -> int sol} except that it is not inert in $O_K$. In conclusion, when $h = 1, 2$ our argument works without the requirement that ``specials'' primes $p_1, \dots, p_{h-1}$ appear in \eqref{eq:rat -> int sol}.  
\end{remark}

\begin{remark}
	\label{rem:cl is proper subgroup in cl+}
	When $a > 1$ and the leading coefficient of $P(t)$ is negative we will need to show that
	\[
		\N_K(x, y) 
	= - p_1 \dots p_{h-1} \prod_{p \in S} p
	\] 
	has an integral solution provided that it has a rational one. A minor modification of the proof of Proposition~\ref{prop:rational imples integral} suffices to see this. If the ideal class group of $K$ coincides with $\ClK$, then there is an element of $O_K$ of norm $-1$ and all principal ideals of $O_K$ lie in the same class inside $\ClK$. Since $[(-1)]$ is clearly the trivial element of $\ClK$, then the same argument as above will give an integral solution satisfying \eqref{eq:integral point}. Multiplying both sides of \eqref{eq:integral point} by $-1$ and exploiting the fact that $-1$ is a norm of an element of $O_K$ gives the desired integral solution. If, however, the ideal class group of $K$ is a subgroup of $\ClK$ of order 2, then the ideal $\gri$ in the proof of Proposition~\ref{prop:rational imples integral} has a generator of negative norm and hence its class in $\ClK$ is no longer the trivial one. Applying the same steps as in the proof of Proposition~\ref{prop:rational imples integral} will now produce a principal fractional ideal $\grI$ of $O_K$ of norm $p_1 \dots p_{h-1} \prod_{p \in S} p$ and such that $[\grI] = [\gri]$. Thus $\grI$ has a generator of negative norm, $m_* + \omega\ell_*$ say. This implies that there are $m_*, \ell_* \in \ZZ$ satisfying $\N_{K/\QQ}(m_* + \omega\ell_*) = - \N(\grI)$ which shows the existence of the desired integral solution.
\end{remark}

\section{A Chebotarev density result} 
\label{sec:Chebotarev}
In this section we prove a preparatory result and its corollary needed to implement Proposition~\ref{prop:rational imples integral} in the proof of Theorem~\ref{thm:Chatelet}. Throughout this section we assume that $h > 2$. The notation fixed in the introduction is still in power. 

\begin{proposition}
	\label{prop:Chebotarev}
	Let $Q(t)$ be an irreducible polynomial of $\ZZ[t]$. Assume that $Q(t)$ is linear or that the discriminant of the splitting field of $Q(t)$ is coprime to $d_K$. Then there exist infinitely many $(h-1)$-tuples $(p_1, \dots, p_{h-1})$ of distinct primes of $\ZZ$ such that
	\begin{enumerate}[label = \emph{(\roman*)}]
		\item $Q(t) \bmod p_j$ is a product of non-repeated linear factors  for $j = 1, \dots, h-1$,
		\item each $(p_j) = \grp_j \grp_j'$ is split in $O_K$,
		\item the non-trivial elements of $\ClK$ are precisely the classes of $\grp_1, \dots, \grp_{h-1}$.
	\end{enumerate}
\end{proposition}

\begin{proof}
	If $Q(t)$ is linear, then (i) is clearly satisfied for any prime not dividing the leading coefficient of $Q(t)$. Thus in this case we only need to check that there are infinitely many $(h-1)$-tuples $(p_1, \dots, p_{h-1})$ satisfying the remaining conditions (ii) and (iii). A minor modification of our method does this. 

	When $Q(t)$ in not linear we need a finer set-up.	Let $c$ be the leading coefficient of $Q(t)$ and let $\alpha$ be a root of $Q(t)$ defined over a fixed algebraic closure $\overline{\QQ}$ of $\QQ$. Let $L = \QQ(\alpha)$ and let $F \subseteq \overline{\QQ}$ be the splitting filed of $Q(t)$. By construction $F/\QQ$ is Galois. 

	Denote by $\ZZ[\frac{1}{c}]$ the localisation of $\ZZ$ with respect to the multiplicative system generated by powers of $c$. By construction it is a Dedekind domain. Let $O$ be the integral closure of $\ZZ[\frac{1}{c}]$ in $L$. By \cite[Prop.~8.1, p.~45]{Neu99} the ring $O$ is also a Dedekind domain. In the case when $Q(t)$ is monic $O$ is simply the ring of integers of $L$. 

	Let $p \in \ZZ$ be a prime not dividing $c$, the conductor of $\ZZ[\frac{1}{c}][\alpha] \subseteq O$ and the discriminant of $L$. The set of primes which fail these assumptions is finite \cite[Prop.~8.4, p.49]{Neu99}. We can now begin translating the conditions (i), (ii), (iii) into conditions on certain associated to $p$ elements of the Galois group of an appropriate number field. 

	Clearly, for any prime $p \nmid c$ of $\ZZ$ the polynomial $Q(t)$ is a product of linear factors mod $p$ if and only if $Q_1(t) = \frac{1}{c}Q(t)$ is a product of linear factors mod $p$. Moreover, by definition $\alpha \in O$ is a root of $Q_1(t)$ and $Q_1(t)$ is monic irreducible polynomial with coefficients in $\ZZ[\frac{1}{c}]$. Then $L = \QQ(\alpha)$ is a separable extension of $\QQ$ given by a primitive element $\alpha \in O$ such that $Q_1(t)$ is its minimal polynomial with coefficients in $\ZZ[\frac{1}{c}]$. By \cite[Prop.~8.3, p.~47]{Neu99} the polynomial $Q_1(t)$ is a product of distinct linear factors mod $p$ if and only if $p$ splits completely in $O$. Moreover, since $p \nmid c$ the prime $p$ is completely split in $O$ if and only if it is completely split in $O_L$ by the commutative diagram
	\[
		\begin{tikzcd}
			&\Spec{O} \arrow[hook]{r} \arrow{d}
			&\Spec{O_K} \arrow{d} \\
			&\Spec{\ZZ[\frac{1}{c}]} \arrow[hook]{r}
			&\Spec{\ZZ}
		\end{tikzcd}
		.
	\]
	If a prime splits completely in any field extension of $L$, then it must split completely in $L$ since the inertia degree and the ramification index are multiplicative in towers of field extensions. Therefore, in order to show that $Q_1(t) \bmod p$ is a product of non-repeated linear factors it suffices to check that $p$ is completely split in $F$. 
	
	The prime $p$ does not divide the discriminant of $L$ and thus it must be coprime to the discriminant of $F$ \cite[Thm~86, p.~97]{Hil98}. The latter implies that $p$ is unramified in $F$ and hence the Artin symbol is well defined on prime ideals of $O_F$ above $p$. We can understand how $p$ decomposes in $O_F$ by looking at these Artin symbols. Let $\grP \mid p$ be a prime ideal of $O_F$. The Artin map associates an element of the Galois group $\Gal(F / \QQ)$ to $\grP$ to it. The prime $p$ splits completely in $F$ if and only if this element is the identity \cite[Cor.~5.21~(iii), p.~107]{Cox89}, i.e.
	\begin{equation}
		\label{eq:F/Q}
		\(\frac{F / \QQ}{\grP}\) 
		= 1.
	\end{equation}
	In conclusion, $Q(t)$ is a product of distinct linear factors mod $p$ if and only if for any prime $\grP$ of $O_F$ above $p$ the associated Artin symbol in the Galois group $\Gal(F / \QQ)$ is the identity.
	
	Once again, by \cite[Cor.~5.21 (iii), p.~107]{Cox89} the condition that $(p) = \grp \grp'$ splits in $O_K$ is equivalent to  
	\begin{equation}
		\label{eq:K/Q}
		\(\frac{K/\QQ}{\grp}\)
		= 1.
	\end{equation}
	
	Let $H$ denote the narrow Hilbert class field of $K$. By definition it is an Abelian extension of $K$ and hence for any prime $\grP$ of $O_H$ lying over a prime $\grp$ of $K$ the Artin symbol associated to $\grP$ only depends on $\grp$. Let $\gra$ be a fixed fractional ideal of $O_K$ for which the Artin symbol is well-defined. There is an isomorphism $\Gal(H/K) \cong \ClK$ given by the Artin map \cite[Thm~8.2, p.~161]{Cox89}. It implies that $[\grp]$ and $[\gra]$ are the same element of $\ClK$ if and only if the corresponding Artin symbols satisfy  
	\begin{equation}
		\label{eq:H/K}
		\(\frac{H/K}{\grp}\)
		= \(\frac{H/K}{\gra}\).
	\end{equation}

	We need to combine all the conditions listed above. To do so, we first examine the structure of the Galois group of a certain compositum. 
	We know that $K/\QQ$ is cyclic, $H/K$ is Abelian and $H/\QQ$ is Galois. Moreover every prime of $\QQ$ that ramifies in $K$ is unramified in $H$ by construction. Thus  
	\begin{equation} 
		\label{eq:Gal(H/QQ)}
		\begin{split}	
			\Gal(H/\QQ) 
			&= \Gal(H/K) \rtimes \Gal(K/\QQ) \\
			&= \ClK \rtimes \ \ZZ/2\ZZ,
		\end{split}	
	\end{equation}	
	by \cite[Prop.~3.7]{KP18}.
	
	Let $M = FH$. Since $(d_K, d_F) = 1$ and $H$ is unramified over finite primes of $K$, then no finite primes simultaneously ramify in $F \cap H$. Thus $F \cap H = \QQ$. By \cite[Thm~1.14, p.~267]{Lan02} the extension $M / \QQ$ is Galois and its Galois group $\Gal(M/ \QQ)$ is the direct product $\Gal(F/\QQ) \times \Gal(H/\QQ)$. This together with \eqref{eq:Gal(H/QQ)} implies that
	\begin{equation}
		\label{eq:big Galois}
		\begin{split}
			\Gal(M / \QQ) 
			&= \Gal(F / \QQ) \times \Gal(H / \QQ) \\
			&= \Gal(F / \QQ) \times \( \Gal(H / K) \rtimes \Gal(K / \QQ) \).
		\end{split}
	\end{equation}
	
	In view of \eqref{eq:F/Q}, \eqref{eq:K/Q}, \eqref{eq:H/K} and \eqref{eq:big Galois} what is left is to take the element 
	\[
		\sigma 
		= \( 1,  \(\frac{H/K}{\gra}\), 1 \) \in \Gal(M/\QQ).	
	\]
	We apply Chebotarev's density theorem \cite[Thm~13.4, p.~545]{Neu99} to the conjugacy class $\left< \sigma \right>$ in $\Gal(M / \QQ)$. By doing so we conclude that there are infinitely many primes $p$ of $\ZZ$ which split completely in $F$. Moreover, they split in $K$ as $(p) = \grp \grp'$ and the classes $[\grp] = [\gra]$ coincide in $\ClK$.
	
	Now let $\gra_1, \dots, \gra_{h-1}$ be fractional ideals of $O_K$ for which the Artin symbol is well defined and such that each of them represents a different non-trivial element of $\ClK$. We apply the above argument with each of the ideals $\gra_1, \dots, \gra_{h-1}$. As a result we get infinitely many $(h-1)$-tuples $(p_1, \dots, p_{h-1})$ of distinct primes of $\ZZ$ satisfying the assumptions of Proposition~\ref{prop:Chebotarev}. This concludes its proof.
\end{proof}

The following corollary of Proposition~\ref{prop:Chebotarev} will be needed in the proof of Theorem~\ref{thm:Chatelet}.

\begin{corollary} \label{cor:Chebotarev}
	Let $P(t) \in \ZZ[t]$ be separable as a polynomial of $\QQ[t]$ with factorisation $P(t) = \prod_{i = 1}^n P_i(t)$ into irreducible polynomials $P_i(t) \in \ZZ[t]$.	Assume that $P_1(t)$ is linear or that the discriminant of the splitting field of $P_1(t)$ is coprime to $d_K$. Then there exist infinitely many $(h-1))$-tuples $(p_1, \dots, p_{h-1})$ of distinct primes of $\ZZ$ satisfying \emph{(ii)} and \emph{(iii)} of Proposition~\ref{prop:Chebotarev} and such that for each $j = 1, \dots, h-1$ there exists a $\lambda_{p_j} \in \ZZ_{p_j}$ with $p_j \dmid P(\lambda_{p_j})$.
\end{corollary}

\begin{proof}
	We apply Proposition~\ref{prop:Chebotarev} with $Q(t) = P_1(t)$ and thus (ii) and (iii) of Proposition~\ref{prop:Chebotarev} are satisfied. For the last part it is enough to show the claim for one of the primes $p_1, \dots, p_{h-1}$. Let $p = p_1$, say. Assume that $p$ is big enough so that it is unramified in the splitting field of $P(t)$ and does not divide the leading coefficient $c$ of $P_1(t)$. Then $P(t) \bmod p$ is separable. We will first show that $P_1(t) \bmod p$ has a simple root which is equivalent to $P(t) \bmod p$ having a simple root since $P(t) \bmod p$ is separable. This is clearly the case if $P_1(t)$ is linear. Assume from now on that $P_1(t)$ is not linear.

	Let $d = \deg P_1(t)$. By (i) of Proposition\ref{prop:Chebotarev} we have
	\[
		P_1(t)
		\equiv c(t - \alpha_1) \dots (t - \alpha_d) \bmod p.
	\]
	We can assume that $\alpha_1, \dots, \alpha_d$ are the reductions mod $p$ of some elements lying in the set $\{0, \dots, p - 1\}$. The same symbols will be used for the preimages of $\alpha_1, \dots, \alpha_d$ in $\{0, \dots, p - 1\}$ where it will be clear from the context which one we mean. 
	
	The factorisation of $P_1(t) \bmod p$ implies that there exists a polynomial $R(t) \in \ZZ[t]$ such that in $\ZZ[t]$ we have the identity
	\[
		P(t)
		= (c(t - \alpha_1) \dots (t - \alpha_d) + pR(t)) P_2(t) \dots P_n(t).
	\]
	Note that none of the linear polynomials $t - \alpha_i \in \ZZ[t]$, $i = 1, \dots, d$ divides $R(t)$. Moreover, none of $t - \alpha_i$ is a factor of $P_i(t) \bmod p$ for $i = 2, \dots, n$. Indeed, the later follows from the separability of $P(t)$ mod $p$ while if $t - \alpha_i$ was a factor of $R(t)$, then $t - \alpha_i$ would have been a factor of $P_1(t)$ as a polynomial in $\ZZ[t]$. This contradicts the assumption that $P_1(t)$ is irreducible in $\ZZ[t]$. 
	
	The above implies that we can choose $\lambda_p \in \Zp$ in the following way,
	\[
		\lambda_p
		=
		\begin{cases}
			\alpha_1
			& \mbox{if } R(\alpha_1) \not\equiv 0 \bmod p, \\
			\alpha_1 + p 
			& \mbox{otherwise}.
		\end{cases}
	\]
	This concludes the proof of Corollary~\ref{cor:Chebotarev}.
\end{proof}

\section{Proof of Theorem~\ref{thm:Chatelet}}
\label{sec:main result}
Recall the statement of Hypothesis~\hyperref[hyp:H_1]{H$_1$}.

\begin{hypothesis h1} \label{hyp:H_1}
	Let $P_1(t), \dots, P_n(t) \in \QQ[t]$ be irreducible polynomials. Let $S$ be a finite set of primes of $\ZZ$ containing $\infty$, all finite primes $p$ where one of the polynomials $P_i$ does not have all of its coefficients in $\Zp$ or has all its coefficients divisible by $p$, and all finite primes less or equal to $\deg\prod_{i = 1}^{n} P_i(t)$. Given elements $\lambda_p \in \Qp$ at finite primes $p \in S$, one may find $\lambda \in \QQ$, integral away from $S$, arbitrarily close to each $\lambda_p$ for the $p$-adic topology for each finite $p \in S$ and arbitrarily big in $\RR$. Moreover, each $P_i(\lambda) \in \QQ$ is a unit in $\Qp$ for all primes $p \not\in S$ except perhaps for one prime $q_i$, where it is a uniformising parameter.
\end{hypothesis h1}

The above conjecture is usually formulated over an arbitrary number field but our application in the proof of Theorem~\ref{thm:Chatelet} only requires its version over $\QQ$. It was noticed by Serre \cite{Ser92} that Hypothesis~\hyperref[hyp:H]{H} implies Hypothesis~\hyperref[hyp:H_1]{H$_1$}. A proof of this fact is available in \cite[\S4]{CTSD94}.

We can now begin the proof of Theorem~\ref{thm:Chatelet}. Notation is as set in the introduction. 

\begin{remark}
	\label{rem:equivalence}
	We have $\N_K(x, y) = x^2 - ay^2$ when $a \equiv 2,3 \bmod 4$ while when $a \equiv 1 \bmod 4$ these two quadratic forms represent $1$ and are equivalent over $\QQ$, over $\Qp$ for all primes $p \le \infty$ and over $\Zp$ when $p \neq 2$. Hence in either case they represent the same elements of the relevant domain.
\end{remark}

Let an adelic point
\begin{equation} 
	\label{eq:adele in BM set}
	(x_p, y_p, t_p) \in \X (\A_\ZZ)^{\mathcal{A}}
\end{equation}
with $(x_\infty, y_\infty, t_\infty)$ in unbounded component of $X(\RR)$ be given. The map described in \eqref{eq:ev composed with inv} is constant on each connected component of $X(\RR)$ \cite[Prop.~8.2.9]{Poo17}. Thus without loss of generality we can assume that $(x_\infty, y_\infty, t_\infty)$ is such that for each $i = 1, \dots, n$ and $t \in \RR$ with $|t| \ge t_\infty$ the sign of $P_i(t)$ does not chance. It is convenient to write $M_p = (x_p, y_p, t_p)$ for the $p$-adic component of the adelic point above. Let $\veps > 0$ be given and let $S$ be a finite set of non-archimedean primes $p$ of $\ZZ$ for which we would like to approximate the $t$-coordinate of $M_p$ with an integral point on $\X$ of a distance at most $\veps$.  

Our argument works in the following way. In view of the inclusions $\X(\ZZ) \subseteq X(\QQ)$ and $\X(\A_{\ZZ}) \subseteq X(\A_{\QQ})$, as sets, we use Hypothesis~\hyperref[hyp:H_1]{H$_1$} to find a point $(x, y, t) \in \QQ \times \QQ \times \ZZ$ lying on $X$ with $|t - t_p|_p < \veps$ for $p \in S$. Fixing $t$ now defines an affine conic with a rational point. This puts us in the set-up of Proposition~\ref{prop:rational imples integral} which, when applied, gives an integral point on $\X$ with the same $t$ coordinate.

We first need to adjust the adele $(M_p)$ and the set $S$. Let $S \subset S_1$ be the finite extension of $S$ obtained by adding $2$, all primes dividing $a$, all primes from Hypothesis~\hyperref[hyp:H_1]{H$_1$} and $h-1$ primes $p_1, \dots, p_{h-1}$ not from the set $S$ which satisfy the assumptions of Proposition~\ref{prop:Chebotarev}.

It is clear that $\prod_{i = 1}^n P_i(t) \neq 0$ defines a dense Zariski open subvariety $U$ of $X$. For each prime $p$ the set $\X(\Zp)$ is open in $X(\Qp)$ with respect to the $p$-adic topology. Thus there exists a point $M_p' = (x_p', y_p', t_p') \in \X(\Zp) \cap U(\Qp)$ with $|t_p - t_p'|_p < \veps$. For our purposes it suffices to assume that each $M_p$ lies in $\X(\Zp) \cap U(\Qp)$. Indeed, for any element $\alpha \in \Br X$ the map \eqref{eq:ev composed with inv} is a locally constant map \cite[Prop.~8.2.9]{Poo17}. Thus $\inv_p(\ev_\alpha(M_p)) = \inv_p(\ev_\alpha(M_p'))$ and hence this small shift is compatible with \eqref{eq:adele in BM set}. 

We claim that for any of $p_1, \dots, p_{h-1}$ there exists a point $M_{p_j}' = (x_{p_j}', y_{p_j}', t_{p_j}') \in \X(\ZZ_{p_j})$ with $p_j \dmid \prod_{i = 1}^n P_i(t_{p_j}')$. By Corollary~\ref{cor:Chebotarev} there is an element $t_{p_j}' \in \ZZ_{p_j}$ such that $p_j \dmid \prod_{i = 1}^n P_i(t_{p_j}')$. Each of the primes $p_j$ is split in $O_K$ and thus $(\frac{a}{p_j}) = 1$. This implies that the congruence $x^2 - ay^2 \equiv \prod_{i = 1}^n P_i(t_{p_j}') \equiv 0 \bmod p_j$ has a non-trivial solution $(x, y)$ over $\FF_{p_j}$. Since Hensel's lemma lifts this to a $\ZZ_{p_j}$-solution Remark~\ref{rem:equivalence} shows that such $M_{p_j}'$ exists. Replace $M_p$ with $M_p'$ for each of the primes $p_1, \dots, p_{h-1}$. Since we chose $p_1, \dots, p_{h-1} \notin S$ we are not replacing any of the components of $(M_p)$ that we want to approximate. Moreover, for any $j = 1, \dots, h-1$ the equality $(\frac{a}{p_j}) = 1$ implies that the map in \eqref{eq:ev composed with inv} is zero on $\X(\ZZ_{p_j})$ for each of the quaternion algebras $(a, P_i(t))$. Therefore \eqref{eq:adele in BM set} is preserved.

The assumption \eqref{eq:adele in BM set} means, in particular, that for each $i = 1, \dots n$ the Hilbert symbols $(a, P_i(t_p))_p$ associated to the quaternion algebras $(a, P_i(t)) \in \Br X$ satisfy $(a, P_i(t_p))_p = 1$ for almost all $p$ and $\prod_{p \le \infty} (a, P_i(t_p))_p = 1$. Thus by \cite[Ch.~III, Thm~4]{Ser73} for each $i = 1, \dots, n$ there exists $c_i \in \QQ^*$ such that $(a, P_i(t_p))_p = (a, c_i)_p$. Therefore for each $i = 1, \dots, n$ and each prime $p$ we have $(a, c_iP_i(t_p))_p = 1$ and hence $c_iP_i(t_p)$ is represented by $x^2 - ay^2$ over $\Qp$. 

We have shown that the affine variety $Y \subset \AA_{\QQ}^{2n + 3}$ given by
\[
	\begin{tikzpicture}[scale = 1.5]	
		\node (Y)    at (-1.7, 0.3) {$Y:$};
		\node (W.eq) at (0.6, 0.6)  {\quad $x_i^2 - ay_i^2 = c_iP_i(t), \quad i = 1, \dots, n$,};
		\node (X.eq) at (0.03, 0) {\quad $x^2 - ay^2 = \prod_{i = 1}^n P_i(t)$};
	\end{tikzpicture}
\]
has a local point with $t = t_p$ for each prime $p$ including $\infty$. 

\begin{lemma}
	\label{lem:rational point on Y}
	There is a point in $Y(\QQ)$ with $t \in \ZZ$ satisfying $|t - t_p|_p < \veps$ when $p \in S_1 \setminus \{\infty\}$.
\end{lemma}

\begin{proof}
	Let $W \subset \AA^{2n + 1}$ be the variety defined over $\QQ$ by
	\[
		W: \quad x_i^2 - ay_i^2 = c_iP_i(t), \quad i = 1, \dots, n.
	\]
	Then $Y$ is the fibre product $Y = W \times_{\AA^{1}_{\QQ}} X$. The change of variables 
	\begin{equation}
		\label{eq:birational map}
		u + \sqrt{a}v
		= (x + \sqrt{a}y)\prod_{i = 1}^{n}\frac{1}{x_i + \sqrt{a}y_i}
	\end{equation}
	defines a birational map over $\QQ$ between $Y$ and $W \times C$, where $C \subset \AA^2_\QQ$ is the affine conic
	\[
		C: \quad u^2 - av^2 = \frac{1}{c_1 \dots c_n}.
	\]
	Moreover, this birational map sends any point on $Y$ with $\prod_{i = 1}^n P_i(t) \neq 0$ to a point on $W \times C$ with the same $t$. For each prime $p$ we now map the local point on $Y$ that we showed to exist above to a local point on $W \times C$. 

	We will show that there exists a rational point on $W \times C$ which can be mapped back to a rational point on $Y$ with $t$ coordinate less than $\veps$ away from $t_p$ for each $p \in S_1$ finite. Before doing so let $S_1 \subseteq S_2$ be the finite extension obtained by adding all primes $p \notin S_1$ for which there is a $c_i$ with non-zero $p$-adic valuation. Applying Hypothesis~\hyperref[hyp:H_1]{H$_1$} now gives $\lambda \in \QQ$ with $|\lambda| > t_\infty$ which satisfies
	\begin{itemize}
		\item $|\lambda - t_p|_p < \veps$ for each $p \in S_2 \setminus \{\infty\}$,
		\item $|\lambda|_p \le 1$ for each $p \not\in S_2$,
	\end{itemize}
	and such that
	\begin{equation}
		\label{eq:hyp h1 output}
		c_iP_i(\lambda) 
		=  \delta_i q_i \prod_{p \in S_2 \setminus \{ \infty \}} p^{l_{p, i}}, \quad i = 1, \dots, n,
	\end{equation}
	with primes $q_1, \dots, q_n \not\in S_2$, exponents $l_{p, i} \in \ZZ$ and $\delta_i \in \{\pm 1\}$. 

	The above conditions imply that, in fact, $\lambda \in \ZZ$. Indeed, without loss of generality we can assume that $\veps \le 1$. Then
		\[
			\left| \lambda \right|_p 
			\le \max \{ \left| \lambda - t_p \right|_p, \left| t_p \right|_p\}
			\le 1 \quad \mbox{for each} \quad p \in S_2 \setminus \{\infty\}.
		\]
	Thus $\lambda \in \Zp$ for each finite prime $p$ and hence it must be an integer. 

	It is now easy to show that for each $p \le \infty$ and each $i = 1, \dots, n $ the rational number $c_i P_i(\lambda)$ is represented by $x^2 - a y^2$ over $\Qp$. For each finite prime $p \in S_2$ we have $|\lambda - t_p|_p < \veps$ and hence $|c_i P_i(\lambda) - c_i P_i(t_p)|_p < \veps$. If $\veps$ is small enough, which we can freely assume, then this together with the fact that $c_iP_i(t_p)$ is represented by $x^2 - a y^2$ over $\Qp$ is enough to claim that $c_iP_i(\lambda)$ is represented as well. On the other hand, when $p \not\in S_2$ and $p \neq q_i$, then $a$ and $c_iP_i(\lambda)$ are both units in $\Zp$. Thus the Hilbert symbol $(a, c_iP_i(\lambda))_p = 1$ which is equivalent to $c_iP_i(\lambda)$ being represented by $x^2 - ay^2$ over $\Qp$. If $a > 0$, then over $\RR$ the form $x^2 - ay^2$ is indefinite and represents every real number, in particular, it represents $c_iP_i(\lambda)$. On the other hand, we already saw that $(a, c_iP_i(t_\infty))_\infty = 1$ which is only possible if $c_iP_i(t_\infty) > 0$ when $a < 0$. Recall that we chose $t_\infty$ so that the polynomials $P_i(t)$ do not chance their sign when $|t| \ge t_\infty$. Since $\lambda$ is in this range we have $c_iP_i(\lambda) > 0$ and therefore it is represented by the positive definite form $x^2 - ay^2$ over $\RR$. 

	We have shown that for each $i = 1, \dots, n$ the product of Hilbert symbols satisfies $\prod_{p \neq q_i} (a, c_i P_i(\lambda))_p = 1$. By the product formula we conclude that $(a, c_i P_i(\lambda))_{q_i} = 1$ and therefore $c_i P_i(\lambda)$ is represented by $x^2 - a y^2$ over each completion of $\QQ$. The Hasse--Minkowski theorem implies now that it must be represented over $\QQ$. In other words, we have shown that for each $i = 1, \dots, n$ there exist $x_i, y_i \in \QQ$ such that 
	\[
		x_i^2 - a y_i^2 
		= c_i P_i(\lambda) 
		= \delta_i q_i \prod_{p \in S_2 \setminus \{ \infty \}} p^{l_{p, i}}
		\neq 0.
	\]
	Thus there is a point in $W(\QQ)$ with $t$-coordinate being $\lambda$. On the other hand, $C$ has points everywhere locally and since affine conics satisfy the Hasse principle we must have $C(\QQ) \neq \emptyset$. Therefore $W \times C$ has a rational point.

	Under the birational map defined by \eqref{eq:birational map} we map the rational point we showed to exist above on $W \times C$ to a $\QQ$-point on $Y$ with the same $t$ coordinate $\lambda$. Since $S_1 \subseteq S_2$ and $|\lambda - t_p|_p < \veps$ for each $p \in S_2 \setminus \{\infty\}$ this concludes the proof.
\end{proof}

Lemma~\ref{lem:rational point on Y} implies that there is a rational point $(x, y, t)$ on the affine variety defined by $x^2 - ay^2 = P(t)$ over $\QQ$. Moreover, $t \in \ZZ$ satisfies $|t - t_p|_p < \veps$ when $p \in S_2$, where $S_2$ is as in the proof of Lemma~\ref{lem:rational point on Y}.
In view of Remark~\ref{rem:equivalence} we get $(x_*, y_*, t) \in X(\QQ)$ with the same $t$. On the other hand, Hypothesis~\hyperref[hyp:H_1]{H$_1$} allows us to freely choose the sign of $t$ so that under the assumptions in Theorem~\ref{thm:Chatelet} we can take $t$ with $P(t) > 0$. Indeed, this is true except when $a > 0$ and the leading coefficient of $P(t)$ is negative in which case there is a negative sign on the far right hand side bellow. Using \eqref{eq:hyp h1 output} we then get that 
\[
	\N_K(x_*, y_*) 
	= \prod_{i = 1}^{n} P_i(t)
	= \prod_{i = 1}^{n} \( q_i \prod_{p \in S_2 \setminus \{ \infty \}} p^{l_{p, i}'}\).
\]
The new exponents $l_{p, i}'$ differ to $l_{p, i}$ only when $p$ divides any of $c_1, \dots, c_n$. 

Recall that $2 \in S_2$. If $a = 1$, then by Lemma~\ref{lem:rational point on Y} we have $|t - t_2|_2 < \veps$ which implies that $P(t) \nequiv 2 \bmod 4$. In this case there is an elementary argument verifying that $P(t)$ is a difference of two integral squares. Indeed, if $P(t)$ is odd, then $P(t) = (\frac{P(t) + 1}{2})^2 - (\frac{P(t) - 1}{2})^2$. On the other hand, if $P(t)$ is even, then it must be divisible by $4$ and similarly we get $P(t) = (\frac{P(t) + 4}{4})^2 - (\frac{P(t) - 4}{4})^2$.

Assume now that $a \neq 1$ and we are not in the case when $a > 0$ and the leading coefficient of $P(t)$ is negative. Recall that we chose $p_1, \dots, p_{h-1}$ in $S_2$ covering all of the non-trivial elements of $\ClK$ in the sense of Proposition~\ref{prop:Chebotarev}. Moreover, we have $\valp (P(t)) = \valp (P(t_p)) = 1$ since $|t - t_p|_p < \veps$ for each $p \in \{p_1. \dots, p_{h-1} \}$. Let $S_3$ be the subset of $S_2 \setminus \{p_1, \dots, p_{h-1}, \infty \}$ consisting of all primes dividing $P(t)$ to odd multiplicity and let $S_* = S_3 \cup \{ q_1, \dots, q_n \}$. Writing $x_* = m/k$ and $y_* = \ell/k$ and redefining $k$ if needed gives
\[
	\N_K(m,\ell) 
	= k^2 p_1 \dots p_{h-1} \prod_{p \in S_*} p.
\]

For any prime $p \in S_*$ the $p$-adic valuation of $\N_K(m,\ell)$ is clearly odd. Thus if $p \in S_*$ is odd, then either $p \mid a$ or $a$ is a quadratic residue modulo $p$ because $N_K(x, y)$ is properly equivalent to $x^2 - ay^2$ over $\Fp$. Hence $p$ is not inert in $O_K$. On the other hand, if $2 \in S_*$ then $2$ is ramified if $a \equiv 2,3 \bmod 4$ and $2$ is split when $a \equiv 1 \bmod 8$. If, however, $a \equiv 5 \bmod 8$ then it is easy to see that the $2$-adic valuation of the left hand side is zero and hence $2 \not\in S_*$. Thus the primes on the right hand side above are not inert in $O_K$. This puts us in position to apply Proposition~\ref{prop:rational imples integral} which ensures that there exist $x, y \in \ZZ$ such that
\begin{equation}
	\label{eq:final integral solution}
	\N_K(x,y)
	= p_1 \dots p_{h-1} \prod_{p \in S_*} p	.
\end{equation}
Multiplying both sides above by a suitable square and redefining $x$ and $y$ if needed gives
\[
	\N_K(x,y)  
	= q_1 \dots q_n \prod_{p \in S_2 \setminus \{ \infty \}} p^{\sum_{i = 1}^n l_{p, i}'}
	= \prod_{i = 1}^n P_i(t) = P(t).
\]
Thus $\X(\ZZ) \neq \emptyset$. What is left is to recall that $|t - t_p|_p < \veps$ for each $p \in S_2 \setminus \{ \infty \}$ by Hypothesis~\hyperref[hyp:H_1]{H$_1$}. Since $S \subset S_2$, then the point $(x, y, t) \in \X(\ZZ)$ fulfils the desired approximation property. 

If $a > 1$ and the leading coefficient of $P(t)$ is negative a small modification is needed. Following the same steps as above will lead us to the situation where instead of \eqref{eq:final integral solution} we need to show that there are $x, y \in \ZZ$ such that 
\begin{equation}
	\label{eq:final negative}
	\N_K(x,y)
	= - p_1 \dots p_{h-1} \prod_{p \in S_*} p,
\end{equation}
provided that there is a rational solution to the above equation. We do so by applying Proposition~\ref{prop:rational imples integral} with Remark~\ref{rem:cl is proper subgroup in cl+} taken into account. This concludes the proof Theorem~\ref{thm:Chatelet}.
\qed

\bibliographystyle{amsalpha}{}
\bibliography{bibliography/references}
\end{document}